\documentclass[12pt]{article}
\usepackage{amsmath}
\usepackage{amsfonts}
\usepackage{amssymb}
\usepackage{graphicx}
\usepackage{amsthm}
\usepackage{enumerate}
\usepackage{tikz}
\usepackage{url,xargs}
\usepackage{hyperref}
\usepackage{textcomp}
\usepackage{soul}

\usetikzlibrary{shapes,arrows,matrix,trees,positioning}

\begin{document}

\title{On the maximum fraction of edges covered by $t$ perfect
  matchings in a cubic bridgeless graph}

\author{L. Esperet \footnote{Laboratoire G-SCOP (Grenoble-INP, CNRS),
    Grenoble, France. Partially supported by the French \emph{Agence
      Nationale de la Recherche} under reference
    \textsc{anr-10-jcjc-0204-01}.} \and
  G. Mazzuoccolo \footnote{Laboratoire G-SCOP (Grenoble-INP, CNRS),
    Grenoble, France. Research supported by a fellowship from the
    European Project ``INdAM fellowships in mathematics and/or
    applications for experienced researchers cofunded by Marie Curie
    actions''. E-mail: {\tt mazzuoccolo@unimore.it}.}}

\maketitle

\newtheorem{theorem}{Theorem} \newtheorem{lemma}[theorem]{Lemma}
\newtheorem{definition}[theorem]{Definition}
\newtheorem{example}[theorem]{Example}
\newtheorem{corollary}[theorem]{Corollary}
\newtheorem{conjecture}[theorem]{Conjecture}
\newtheorem{observation}[theorem]{Observation}
\newtheorem{problem}[theorem]{Problem}

\newcommand{\Ha}{\mathring{H}}
\def\scc{\mathop{\mathrm{scc}}\nolimits}

\begin{abstract} \noindent 
A conjecture of Berge and Fulkerson (1971) states that every
cubic bridgeless graph contains 6 perfect matchings covering each edge
precisely twice, which easily implies that every cubic bridgeless graph
has three perfect matchings with empty intersection (this weaker statement was
conjectured by Fan and Raspaud in 1994). Let $m_t$ be the supremum of all
reals $\alpha\le 1$ such that for every cubic 
bridgeless graph $G$, there exist $t$ perfect matchings of $G$
covering a fraction of at least $\alpha$ of the edges of $G$. It is
known that the Berge-Fulkerson conjecture is equivalent to the statement
that $m_5=1$, and implies that $m_4=\tfrac{14}{15}$ and
$m_3=\tfrac45$. In the first part of this paper, we show 
that $m_4=\tfrac{14}{15}$ implies $m_3=\tfrac45$, and $m_3=\tfrac45$ implies the
Fan-Raspaud conjecture, strengthening a recent result of
Tang, Zhang, and Zhu. In the second part of the paper, we prove that for
any $2\le t \le 4$ and for any real
$\tau$ lying in some appropriate interval, deciding whether a fraction
of more than (resp. at least) $\tau$ of the edges of a given cubic
bridgeless graph can be covered by $t$ perfect matching is an
NP-complete problem. This resolves a conjecture of Tang, Zhang, and Zhu.
\end{abstract}

\noindent \textit{Keywords: Berge-Fulkerson conjecture, perfect
  matchings, cubic graphs.}

\section{Introduction}\label{sec:intro}

Most of the notation used in this paper is standard (see
\cite{BoMu} or~\cite{Diestel} for any undefined terminology).  A
\emph{perfect matching} of a graph $G$ is a spanning subgraph of $G$
in which each vertex has degree precisely one. In this paper we will
only deal with \emph{cubic bridgeless} graphs, that is graphs in which
each vertex has degree 3 and such that each component is
2-edge-connected. We are interested in the following conjecture of
Berge and Fulkerson~\cite{Ful}.

\begin{conjecture}[Berge-Fulkerson, 1971]\label{Berge-Fulkerson}
If $G$ is a cubic bridgeless graph, then there exist six perfect matchings
of $G$ such that each edge of $G$ belongs to exactly two of them.  
\end{conjecture}

This conjecture is equivalent to the existence of five perfect
matchings of $G$ such that any three of them have empty
intersection. Therefore, a weaker statement is the following
conjecture of Fan and Raspaud~\cite{FanRas}.

\begin{conjecture}[Fan-Raspaud, 1994]\label{Fan-Raspaud}
If $G$ is a cubic bridgeless graph, then there exist three perfect
matchings $M_1$, $M_2$ and $M_3$ of $G$ such that $M_1 \cap M_2 \cap
M_3 = \emptyset$.  
\end{conjecture}

In this paper we will also consider the following middle step between these two
conjectures. 

\begin{conjecture}\label{4PMs}
If $G$ is a cubic bridgeless graph, then there exist four perfect
matchings such that any three of them have empty intersection.  
\end{conjecture}

Following the notation introduced in \cite{KaiKraNor} we define
$m_t(G)$ as the maximum fraction of edges in $G$ that can be covered
by $t$ perfect matchings and we define $m_t$ as the infimum of
$m_t(G)$ over all cubic bridgeless graphs $G$.

The second author~\cite{Maz} proved that the Berge-Fulkerson conjecture is
equivalent to 
the conjecture that the edge-set of every cubic bridgeless graph can be
covered by 5 perfect matchings, i.e. $m_5=1$. Kaiser,
Kr\'al', and Norine~\cite{KaiKraNor} proved that the infimum $m_2$ is
a minimum, attained by the Petersen graph (i.e. $m_2=\tfrac35$), and
Patel~\cite{Pat} 
conjectured that the values of $m_3$ and $m_4$ are also attained
by the Petersen graph. In other words:

\begin{conjecture}[Patel, 2006]\label{Pat3}
$m_3=\tfrac45$.
\end{conjecture}

\begin{conjecture}[Patel, 2006]\label{Pat4}
$m_4=\tfrac{14}{15}$.
\end{conjecture}

Patel proved~\cite{Pat} that the Berge-Fulkerson conjecture implies
Conjectures~\ref{Pat3} and~\ref{Pat4}. These implications
are summed up in Figure~\ref{fig:implicationsA}.

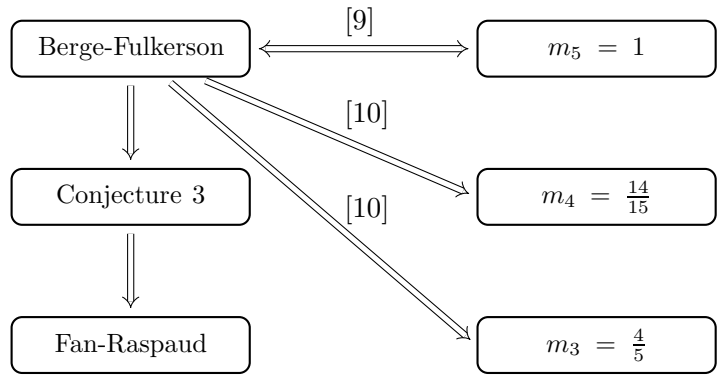
\begin{figure}[ht]
\centering
\begin{tikzpicture}[node distance=1.2cm, auto,
implies/.style={double,double equal sign distance,shorten <=3pt,
           shorten >=3pt,-implies},boite/.style={
           rectangle,
           rounded corners,
           draw=black, thick,
           text width=7em,
           minimum height=1.8em,
           text centered,execute at begin node=\footnotesize}]
 \node[boite] (berge) {Berge-Fulkerson};;
 \node[boite, below=of berge]
 (conj3) {Conjecture~\ref{4PMs}};
\node[boite, below=of conj3]
 (fan) {Fan-Raspaud} ;
\node[boite, right=3cm of berge] (m5) {$m_5=1$};
\node[boite, right=3cm of conj3] (m4) {$m_4=\tfrac{14}{15}$};
\node[boite, right=3cm of fan] (m3) {$m_3=\tfrac45$};
\node[right=1.1cm of berge,yshift=10pt] (c1) {\small \cite{Maz}};
\node[right=1.1cm of conj3,yshift=31pt] (c2) {\small \cite{Pat}};
\node[right=1.1cm of conj3,yshift=-5pt] (c3) {\small \cite{Pat}};

\draw [implies] (berge)--(conj3);
\draw [implies,implies-implies] (berge)--(m5);
\draw [implies] (berge)--(m4.west);
\draw [implies] (berge)--(m3.west);
\draw [implies] (conj3)--(fan);  

\end{tikzpicture}

\caption{A diagram of the implications presented in Section
  \ref{sec:intro}}\label{fig:implicationsA}
\end{figure}

Tang, Zhang, and Zhu~\cite{TanZhaZhu} recently showed that weighted
versions of Conjectures~\ref{Pat3} and~\ref{Pat4} imply the
Fan-Raspaud conjecture: for instance, they prove that if for any cubic
bridgeless graph $G$ and any edge-weighting of $G$, a weighted version
of $m_3(G)$ is at least $\tfrac{4}{5}$, then the Fan-Raspaud
conjecture holds. In Section~\ref{sec:impl}, we show that it is enough to consider only the
edge-weighting where every edge has weight 1. More precisely, we show
that $m_4=\tfrac{14}{15}$ implies Conjecture~\ref{4PMs} and
$m_3=\tfrac{4}{5}$ implies the Fan-Raspaud conjecture. We also show
that $m_4=\tfrac{14}{15}$ implies $m_3=\tfrac{4}{5}$.

Tang, Zhang, and Zhu~\cite{TanZhaZhu} conjectured that
for any real $\tfrac45<\tau \le 1$, determining whether a cubic bridgeless
graph $G$ satisfies $m_3(G)\ge \tau$ is an NP-complete problem. In
Section~\ref{sec:np} we prove this conjecture together with
similar statements for $m_2(G)$ and $m_4(G)$.

\section{Main results}\label{sec:impl}

Given two cubic graphs $G$ and $H$ and two edges $xy$ in $G$ and $uv$
in $H$, the \emph{glueing}, or \emph{2-cut-connection}, of $(G,x,y)$ and $(H,u,v)$ is the
graph obtained from $G$ and $H$ by removing edges $xy$ and $uv$, and
connecting $x$ and $u$ by an edge, and $y$ and $v$ by an edge.  We
call these two new edges the \emph{clone edges of $xy$ or $uv$} in the
resulting graph.  Note that if $G$ and $H$ are cubic and bridgeless, then
the resulting graph is also cubic and bridgeless. In the present paper
$H$ will always be $K_4$ or the Petersen graph, which are both
arc-transitive (for any two pairs of adjacent vertices $u_1,u_2$ and
$v_1,v_2$, there is an automorphism that maps $u_1$ to $v_1$ and $u_2$
to $v_2$). In this case the choice of $uv$ and the order of 
each pair $(x,y)$ and $(u,v)$ are not relevant, so we simply say that
\emph{we glue $H$ on the edge $xy$ of $G$}. 

In what follows, we will need to glue several graphs on each edge of a
given graph $G$. This has to be understood as follows: given copies
$H_1,\ldots,H_k$ ($k\ge 2$) of $K_4$ or the Petersen graph, \emph{glueing
  $H_1,\ldots,H_k$ on the edge $e$ of $G$} means glueing $H_k$ on some
clone edge of $e$ in the glueing of $H_1,\ldots,H_{k-1}$ on the
edge $e$ of $G$ (see Figure~\ref{fig:example}).


\begin{figure}[h]
\centering
\includegraphics[width=8cm]{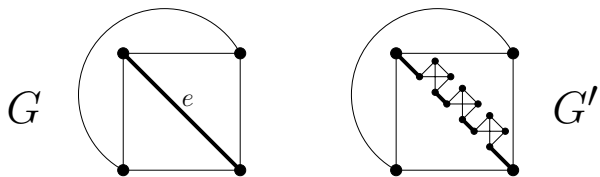}
\caption{The graph $G'$ is obtained by glueing 3 copies of $K_4$ on the
  edge $e$ of $G$. The clone
  edges of $e$ in $G'$ are drawn as thick edges.}\label{fig:example}
\end{figure}

Note that each perfect matching in the graph $G'$ resulting from the
glueing of $H_1,\ldots,H_k$ ($k\ge 1$) on some edge $e$ of $G$ either
contains all clone edges of $e$, or none of them (since each pair of
such edges forms a 2-edge-cut in $G'$, and the intersection of every
perfect matching with an edge-cut has the same parity as the
edge-cut). It follows that each perfect matching $M'$ of $G'$
corresponds to a perfect matching $M$ of $G$ and perfect matchings
$M_i$ of $H_i$, for each $1\le i\le k$. We call each of these perfect
matchings the \emph{restriction} of $M$ to $G,H_1,\ldots,H_k$,
respectively.




\medskip

The following is a well-known property of the Petersen graph.


 

\begin{lemma}\label{lem:petersen}
The Petersen graph $P$ has exactly 6 distinct perfect matchings. Moreover the
following properties hold:
\begin{itemize}
\item Every edge of $P$ is covered by exactly two distinct perfect matchings;

 \item Every two distinct perfect matchings of $P$ intersect in exactly one edge (and therefore cover exactly 9 edges);
 
 \item Every three distinct perfect matchings of $P$ cover exactly 12 edges;
 \item Every four distinct perfect matchings of $P$ cover exactly 14 edges.
\end{itemize}
\end{lemma}

It was proved in~\cite{TanZhaZhu} that a fractional version of
Conjecture \ref{Pat3} ($m_3=\tfrac45$) implies the Fan-Raspaud
conjecture. We strengthen this result by showing that not only the
integral conjecture itself implies the Fan-Raspaud conjecture, but
it implies an even stronger statement (which will be needed in what
follows).

\begin{theorem}\label{the:m3_FR}
  Conjecture \ref{Pat3} implies that any cubic bridgeless graph $G$
  has three perfect matchings
  $M_1,M_2,M_3$ such that $M_1\cap M_2 \cap M_3=\emptyset$ and $|M_1 \cup
  M_2 \cup M_3| \ge \tfrac45\, |E(G)|$. In
  particular, Conjecture \ref{Pat3} implies the Fan-Raspaud Conjecture.
\end{theorem}
\begin{proof}
Let $G$ be a cubic bridgeless graph, and let $G'$ be the graph
obtained by glueing $|E(G)|$ copies of the 
Petersen graph on each edge of $G$ (see Figure
\ref{fig:glueingP}). The number of edges of $G'$ can be easily
computed as $15|E(G)|^2+|E(G)|$. If Conjecture \ref{Pat3} holds, then
$G'$ has a set ${\cal M}'=\{M'_1,M'_2,M'_3\}$ of perfect matchings
such that $M'_1\cup M'_2 \cup M'_3$ contains at least
$\tfrac45\,|E(G')|= 12|E(G)|^2+\tfrac45\,|E(G)|$ edges of
$G'$. Let ${\cal M}=\{M_1,M_2,M_3\}$ be the restriction of ${\cal M}'$
to $G$. We now prove that $|M_1 \cup
  M_2 \cup M_3| \ge \tfrac45\, |E(G)|$ and $M_1\cap M_2 \cap M_3=\emptyset$.

\begin{figure}[ht]
\centering
\includegraphics[width=9cm]{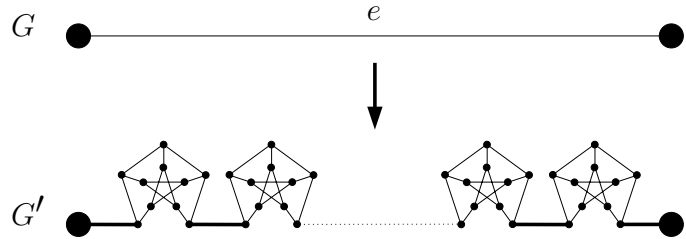}
\caption{$G'$ is obtained by glueing multiple copies of the Petersen
  graph on each edge of $G$. The thick edges are the clone edges
  of $e$ in $G'$.}\label{fig:glueingP}
\end{figure}

Note that by Lemma~\ref{lem:petersen}, at most 12 edges are covered
by the restriction of ${\cal M}'$ to each of the $|E(G)|^2$ copies of
the Petersen graph glued on the edges of $G$. Since $M'_1\cup M'_2 \cup M'_3$
contains at least $12|E(G)|^2+\tfrac45\,|E(G)|$ edges of
$G'$, it follows that $|M_1 \cup
  M_2 \cup M_3| \ge \tfrac45\, |E(G)|$. 

By Lemma~\ref{lem:petersen}, for each edge $e$ of
$G$ lying in 
$M_1\cap M_2\cap M_3$, at most 9 edges are covered
by the restriction of ${\cal M}'$ to each of the $|E(G)|$ copies of
the Petersen graph glued on $e$. If $M_1\cap M_2\cap
  M_3\neq \emptyset$, this implies that $M'_1\cup M'_2 \cup M'_3$
  contains at most  
$12 |E(G)|(|E(G)|-1)+9|E(G)|+|E(G)| = 12 |E(G)|^2
 - 2 |E(G)| $ edges, a contradiction. 
\end{proof}

Let $\mathcal M$ be a set of perfect matchings of a graph $G$.  For
$e\in E(G)$, we define the \emph{depth} of $e$ in $\mathcal M$, denoted by
$dp_e({\cal M})$, as the number of perfect matchings of $\cal M$
containing $e$.  The \emph{edge-depth} of $\cal M$, denoted by
$dp({\cal M})$, is the maximum depth of an edge of $G$ in $\cal M$.
The same proof as that of Theorem~\ref{the:m3_FR}, considering four
perfect matchings instead of three, shows the following connection
between Conjecture \ref{Pat4} ($m_4=\tfrac{14}{15}$) and Conjecture
\ref{4PMs}.

\begin{theorem}\label{the:m4_4PMs}
  Conjecture \ref{Pat4} implies that any cubic bridgeless graph $G$
  has a set ${\cal M}=\{M_1,M_2,M_3,M_4\}$ of four perfect matchings
  such that $dp({\cal M})\le 2$  and $|M_1 \cup
  M_2 \cup M_3 \cup M_4| \ge \tfrac{14}{15}\, |E(G)|$. In
  particular, Conjecture \ref{Pat4} implies Conjecture \ref{4PMs}.
\end{theorem}

We now use Theorem~\ref{the:m4_4PMs} to show that $m_4=\tfrac{14}{15}$
implies $m_3=\tfrac45$.

\begin{theorem}\label{the:141545}
Conjecture \ref{Pat4} implies Conjecture \ref{Pat3}.
\end{theorem}
\begin{proof}
Assume that Conjecture \ref{Pat4} holds and let $G$ be any cubic
bridgeless graph. By Theorem~\ref{the:m4_4PMs}, $G$ has a set ${\cal
  M}=\{M_1,M_2,M_3,M_4\}$ of four perfect matchings such that
$dp({\cal M})\le 2$ and $|M_1 \cup M_2 \cup M_3 \cup M_4| \ge
\tfrac{14}{15}\, |E(G)|$. For $i\ge 0$, let $\varepsilon_i$ be the
fraction of edges of $G$ covered precisely $i$ times. Consider the set
$E_1$ of edges covered exactly once by $\cal M$, and remove from $\cal
M$ the perfect matching, say $M_1$, containing the smallest number of
edges of $E_1$. Then the fraction of edges of $G$ covered by $M_2\cup
M_3 \cup M_4$ is at least $\frac{3}{4}\varepsilon_1 + \varepsilon_2$.

Since $dp({\cal
  M})\le 2$ and every perfect matching contains a third of the edges,
we have $\varepsilon_1+2\varepsilon_2=\tfrac43$. Combining this with
the assumption that $\varepsilon_1+\varepsilon_2\ge \tfrac{14}{15}$,
we obtain $\tfrac34\,\varepsilon_1+\varepsilon_2\ge \tfrac45$, which
concludes the proof.
\end{proof}

\begin{figure}[ht]
\centering
\begin{tikzpicture}[node distance=1.2cm, auto,
implies/.style={double,double equal sign distance,shorten <=3pt,
           shorten >=3pt,-implies},boite/.style={
           rectangle,
           rounded corners,
           draw=black, thick,
           text width=7em,
           minimum height=1.8em,
           text centered,execute at begin node=\footnotesize}]
 \node[boite] (berge) {Berge-Fulkerson};;
 \node[boite, below=of berge]
 (conj3) {Conjecture~\ref{4PMs}};
\node[boite, below=of conj3]
 (fan) {Fan-Raspaud} ;
\node[boite, right=3cm of berge] (m5) {$m_5=1$};
\node[boite, right=3cm of conj3] (m4) {$m_4=\tfrac{14}{15}$};
\node[boite, right=3cm of fan] (m3) {$m_3=\tfrac45$};
\node[right=0.6cm of conj3,yshift=10pt] (c1) {\footnotesize Theorem~\ref{the:m4_4PMs}};
\node[right=0.6cm of fan,yshift=10pt] (c2) {\footnotesize Theorem~\ref{the:m3_FR}};
\node[below=0.3cm of m4,xshift=30pt] (c3) {\footnotesize Theorem~\ref{the:141545}};

\draw [implies] (berge)--(conj3);
\draw [implies,implies-implies] (berge)--(m5);
\draw [implies] (m5)--(m4);
\draw [implies,very thick] (m4)--(m3);
\draw [implies] (conj3)--(fan); 
 \draw [implies,very thick] (m4)--(conj3);
\draw [implies,very thick] (m3)--(fan);

\end{tikzpicture}

\caption{New diagram of implications after
  Section~\ref{sec:impl}}\label{fig:implicationsB}

\end{figure}
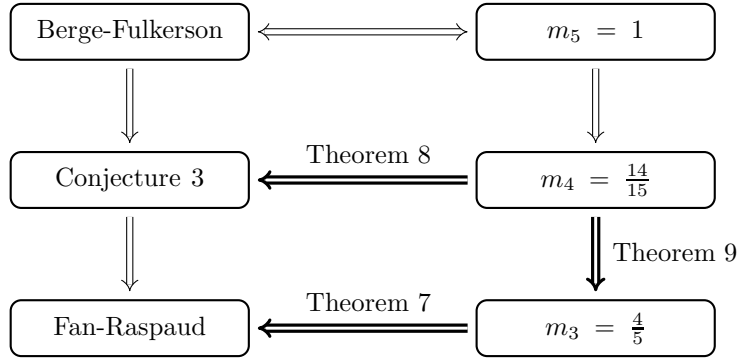

We summarize in Figure~\ref{fig:implicationsB} the implications which
follow from the 
results of this section. An important remark is that in order to prove
that a given cubic 
bridgeless graph $H$ satisfies $m_3(H)\ge \tfrac45$, in
Theorem~\ref{the:141545} we really use the assumption that $m_4(G)\ge
\tfrac{14}{15}$ \emph{for every cubic bridgeless graph $G$}. We do not
know how to prove the stronger statement that if
$m_4(G)=\tfrac{14}{15}$, then $m_3(G)\ge \tfrac45$. In the remainder
of this section, we provide weaker bounds relating $m_k(G)$
and $m_{k+1}(G)$, for any $k\ge 2$.

\begin{theorem}
If $G$ is a cubic bridgeless graph, then $m_2(G)\ge \tfrac12
\,m_3(G)+\tfrac16$.
\end{theorem}

\begin{proof}
Let $M_1,M_2,M_3$ be perfect matchings of $G$ covering a fraction of
$m_3(G)$ of the edges of $G$. For $i\ge 0$, let $\varepsilon_i$ be the
fraction of edges of $G$ covered precisely $i$ times by
$M_1,M_2,M_3$. We have $\mathbf{(1)}$
$\varepsilon_1+\varepsilon_2+\varepsilon_3=m_3(G)$ and $\mathbf{(2)}$
$\varepsilon_1+2\varepsilon_2+3\varepsilon_3=1$. The combination
$\tfrac12 \,\mathbf{(1)}+\tfrac16 \,\mathbf{(2)}$ gives
$\tfrac23\,\varepsilon_1+\tfrac56\,\varepsilon_2+\varepsilon_3=
\tfrac12\,m_3(G)+\tfrac16$. We remove from $M_1,M_2,M_3$ the perfect
matching, say $M_1$, containing the least number of edges covered
precisely once by $M_1,M_2,M_3$. It follows that the fraction of edges of
$G$ covered by $M_2,M_3$ is at least
$\tfrac23\,\varepsilon_1+\varepsilon_2+\varepsilon_3\ge
\tfrac12\,m_3(G)+\tfrac16$.
\end{proof}

\begin{theorem}\label{th:mk}
If $G$ is a cubic bridgeless graph and $k\ge 4$ is an integer, then
$m_{k-1}(G)\ge \tfrac{k-1}k \,m_k(G)+\tfrac1{3k}$.
\end{theorem}

\begin{proof}
Let $M_1,\ldots,M_k$ be perfect matchings of $G$ covering a fraction
of $m_k(G)$ of the edges of $G$. We say that a vertex of $G$ has type
$(x,y,z)$ if the three edges incident to $v$ are covered $x$, $y$, and
$z$ times (respectively) by $M_1,\ldots,M_k$. Let $a,b,c,d,e,f$ be the
fractions of vertices of type $(0,0,k)$, $(0,1,k-1)$, $(0,\ge 2,\ge
2)$, $(1,1,k-2)$, $(1,\ge 2,\ge 2)$, and $(\ge 2,\ge 2,\ge 2)$,
respectively. Any vertex of $G$ is of one of these six types, so we
have $\mathbf{(1)}$ $a+b+c+d+e+f=1$. Let $n$ and $m$ denote the number of vertices and edges of $G$, respectively. Then by definition of $m_k(G)$, the number of edges covered by $M_1,\ldots,M_k$ is
$m \cdot m_k(G)=\tfrac12(an+2bn+2cn+3dn+3en+3fn)$. Since $G$ is cubic, $m=\tfrac32n$ and
so, $\mathbf{(2)}$ $\tfrac{a}3+\tfrac{2b}3+\tfrac{2c}3+d+e+f=m_k(G)$.
Similarly, the fraction of edges covered precisely once by
$M_1,\ldots,M_k$ is equal to $\tfrac{b}3+\tfrac{2d}3+\tfrac{e}3$, so
we can remove one of the perfect matchings, say $M_1$, in such way
that $M_2,\ldots,M_k$ cover a fraction of at least
$m_k(G)-\tfrac1k\,(\tfrac{b}3+\tfrac{2d}3+\tfrac{e}3)$ of the edges of
$G$. Note that the combination $\mathbf{(2)}-\tfrac1{3} \,\mathbf{(1)}$ gives
$m_k(G)-\tfrac1{3}=\tfrac{b}3+\tfrac{c}{3}+\tfrac{2d}{3}+\tfrac{2e}{3}+\tfrac{2f}{3} \geq \tfrac{b}3+\tfrac{2d}3+\tfrac{e}3$, 
so it follows that $m_{k-1}(G)\ge \tfrac{k-1}k \,m_k(G)+\tfrac1{3k}$.
\end{proof}

In particular, this shows that any cubic bridgeless graph $G$ whose
edge-set can be covered by 4 perfect matchings satisfies $m_3(G)\ge
\tfrac{5}{6}$. It follows that the conjecture stating that every cubic
bridgeless graph $G$ satisfies $m_3(G)\ge \tfrac45$ only needs to be
verified for graphs whose edge-set cannot be covered by 4 perfect
matchings (some results on this class of graphs can be found
in~\cite{EspMaz} and~\cite{Hag}).

\smallskip

Recall that Kaiser, Kr\'al', and Norine~\cite{KaiKraNor} proved that every cubic
bridgeless graph $G$ satisfies $m_2(G)\ge \tfrac35$. An interesting
problem is to characterize graphs for which equality holds (the
Petersen graph is an example). The next theorem implies that, again,
these graphs are such that their edge-set cannot be covered by 4 perfect
matchings (indeed, if the edge-set of a cubic bridgeless graph $G$ can be covered by 4 perfect matchings, then $m_4(G)=1$, and
Theorem~\ref{th:m2m4} below implies that $m_2(G)\ge \tfrac{11}{18}>
\tfrac35$).

\begin{theorem}\label{th:m2m4}
If $G$ is a cubic bridgeless graph, then
$m_{2}(G)\ge \tfrac5{12} \,m_4(G)+\tfrac7{36}$.
\end{theorem}

\begin{proof}
Let $M_1,\ldots,M_4$ be perfect matchings of $G$ covering a fraction of
$m_4(G)$ of the edges of $G$. Let $a,b,c,d,e,f$ be the fractions of
vertices defined in the proof of Theorem~\ref{th:mk}, with
$k=4$. Observe that $e=f=0$, thus we have $\mathbf{(1)}$ $a+b+c+d=1$ and
$\mathbf{(2)}$ $\tfrac{a}3+\tfrac{2b}3+\tfrac{2c}3+d=m_4(G)$. Let $\varepsilon_i$
be the fraction of edges of $G$ covered precisely $i$ times by
$M_1,M_2,M_3,M_4$.

Note that after the removal of two perfect matchings $M_i$ and $M_j$
($i\ne j$) from $M_1,\ldots,M_4$, the edges that were covered only by
$M_i$, or only by $M_j$, or only by $M_i$ and $M_j$ are not covered
anymore. If we sum the fractions of edges of these three types, for any
of the six pair $\{i,j\}$ we obtain $3\varepsilon_1+\varepsilon_2$
(every edge covered exactly once is counted three times). Note that
$\varepsilon_1=\tfrac{b}{3}+\tfrac{2d}3$ and
$\varepsilon_2=\tfrac{2c}3+\tfrac{d}3$, so
$3\varepsilon_1+\varepsilon_2=b+\tfrac{2c}3+\tfrac{7d}3$. Thus we can choose a
pair $\{i,j\}$, such that after the removal of $M_i$ and $M_j$ from
$M_1,\ldots,M_4$, a fraction of at least $m_4(G)-\tfrac16(b+\tfrac{2c}3+\tfrac{7d}3)$
of the edges of $G$ is still covered. The combination $\tfrac72 \,\mathbf{(2)}-\tfrac7{6}
\,\mathbf{(1)}$ gives
$\tfrac72\,m_4(G)-\tfrac7{6}=\tfrac{7b}6+\tfrac{7c}6+\tfrac{7d}3 \geq b+\tfrac{2c}3+\tfrac{7d}3 $, so it follows
that $m_2(G)\ge \tfrac5{12} \,m_4(G)+\tfrac7{36}$.
\end{proof}

\section{Complexity}\label{sec:np}

A cubic bridgeless graph $G$ satisfies $m_3(G)=1$
if and only if it is 3-edge-colorable, and deciding this is a
well-known NP-complete problem (see \cite{Hol}). It
was proved in~\cite{EspMaz} that deciding whether a cubic bridgeless graph $G$ satisfies 
$m_4(G)=1$ is also an NP-complete problem.

In this section, we prove that determining whether $m_2(G)$, $m_3(G)$,
and $m_4(G)$ are more than any given constant (lying in some
appropriate interval) is also an NP-complete problem. In the case of
$m_3(G)$ (see Theorem~\ref{th:m3}), this solves a conjecture of Tang,
Zhang, and Zhu~\cite{TanZhaZhu}. 


\begin{theorem}\label{th:m2}
For any constant $ \frac{3}{5} < \tau <\frac23$, deciding whether a
cubic bridgeless graph $G$ satisfies $m_2(G)> \tau$ (resp. $m_2(G)\ge
\tau$) is an NP-complete problem.
\end{theorem}

\begin{proof}
The proof proceeds by reduction from the 3-edge-colorability of cubic
bridgeless graphs, which is a well-known NP-complete
problem~\cite{Hol}. Note that our problem is clearly in NP, since any set of two
perfect matchings whose union covers a fraction of more than (resp. at
least) $\tau$ of
the edges is a certificate that can be checked in polynomial time.

For any cubic bridgeless graph $G$ with $m$ edges,
we construct (in polynomial time) a cubic bridgeless graph $G'$ (of size
polynomial in $m$), such that $G$ is 3-edge-colorable if and only if
$m_2(G')> \tau$ (resp. $m_2(G')\ge \tau$).

Let $a=\lfloor \tfrac{m}2\rfloor+1$ and
$b=\lfloor a\cdot \tfrac{4-6\tau}{15\tau-9}\rfloor$. It
can be checked that $$\tau < \tfrac{4a+9b+2/3}{6a+15b+1}<
\tau+\tfrac{2a}{m(6a+15b+1)}.$$

The graph
$G'$ is obtained from $G$ by glueing $a$ copies of $K_4$
and $b$ copies of the Petersen graph on every edge of $G$. Note that
$G'$ has precisely $m(6a+15b+1)$ edges.

Assume first that $G$ is 3-edge-colorable. Then $G$ has two perfect
matchings $M_1$ and $M_2$ such that $M_1\cap M_2=\emptyset$. We construct two perfect matchings $M_1'$ and $M_2'$ of $G'$ by combining
$M_1$ and $M_2$, respectively, with perfect matchings of each copy of
$K_4$ and the Petersen graph glued on the edges of $G$. Since no edge
of $G$ is contained both in $M_1$ and $M_2$, these two perfect matchings can be
combined with perfect matchings covering precisely 4 edges of each
copy of $K_4$ (see Figure~\ref{fig:m2}, left and center) and 9 edges
of each copy of the Petersen
graph (by Lemma~\ref{lem:petersen}). It follows that $G'$ contains two
perfect matchings $M_1'$ and $M_2'$ covering $m(4a+9b)+\tfrac23\,m$
edges of $G'$. Therefore, we have $m_2(G')\ge \tfrac{4a+9b+2/3}{6a+15b+1} >
\tau$, as desired.

\begin{figure}[ht]
\centering
\includegraphics[width=12cm]{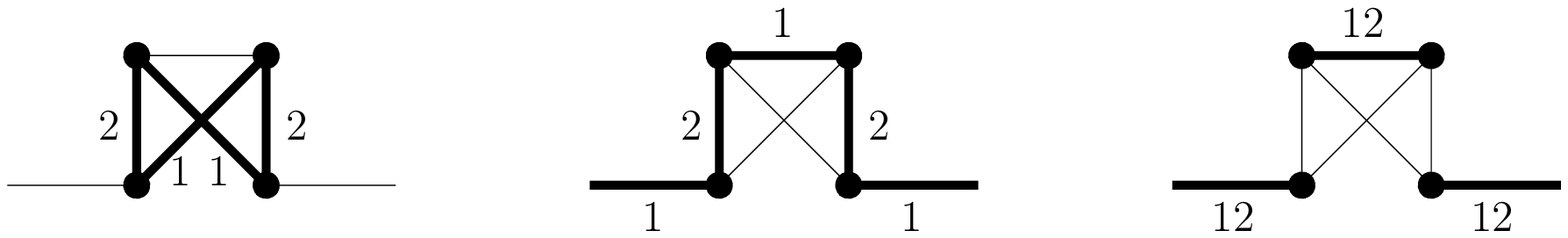}
\caption{Perfect matchings covering a copy of $K_4$ glued on some edge
of $G$. Edges are labelled $i$ if they are covered by $M_i'$ in $G'$.}\label{fig:m2}
\end{figure}

Assume now that $G$ is not 3-edge-colorable, and take any two perfect
matchings $M_1'$ and $M_2'$ of $G'$. Let $M_1$ and $M_2$ be the
restriction of $M_1'$ and $M_2'$ to $G$, respectively. Since $G$ is
not 3-edge-colorable, $M_1$ and $M_2$ have non-empty intersection, so
$G$ has an edge $e$ lying both in $M_1$ and $M_2$. We consider the
union of the restriction of $M_1'$ and the restriction of $M_2'$ to
each copy of $K_4$ or the Petersen graph glued on the edges of $G$. As
before, this union covers at most 4 edges in each copy of $K_4$, and
at most 9 edges in each copy of the Petersen graph. However, it can be
checked that since $e$ is contained in $M_1\cap M_2$, at most 2 edges
can be covered in each copy of $K_4$ glued on $e$ (see
Figure~\ref{fig:m2}, right). It follows that $M_1' \cup M_2'$ covers
at most $m(4a+9b+\tfrac23)-2a$ edges of $G'$. Therefore, $m_2(G')\le
\tfrac{4a+9b+2/3}{6a+15b+1}-\tfrac{2a}{m(6a+15b+1)} < \tau$, as
desired.
\end{proof}

\begin{theorem}\label{th:m3}
For any constant $ \frac{4}{5} < \tau <1$, deciding whether a cubic
bridgeless graph $G$ satisfies $m_3(G)> \tau$ (resp. $m_3(G)\ge \tau$)
is an NP-complete problem.
\end{theorem}

\begin{proof}
The proof follows the lines of the proof of Theorem~\ref{th:m2}. Here
we define $a=\lfloor \tfrac{3m}2\rfloor+1$, where $m$ is the number of edges of $G$, and $b=\lfloor a\cdot
\tfrac{6-6\tau}{15\tau-12}\rfloor$. It can be checked
that $$\tau < \tfrac{6a+12b+1}{6a+15b+1}<
\tau+\tfrac{2a}{m(6a+15b+1)}.$$ The graph $G$ is 3-edge-colorable if
and only if it contains 3 perfect matchings with pairwise empty
intersection. It follows that as above, if $G$ is 3-edge-colorable we
can find 3 perfect matchings of $G'$ such that their restrictions
cover 6 edges in each copy of $K_4$ (see Figure~\ref{fig:m3}, left)
and $12$ edges in each copy of the Petersen graph (by
Lemma~\ref{lem:petersen}) glued on the edges of $G$. 

\begin{figure}[ht]
\centering
\includegraphics[width=8cm]{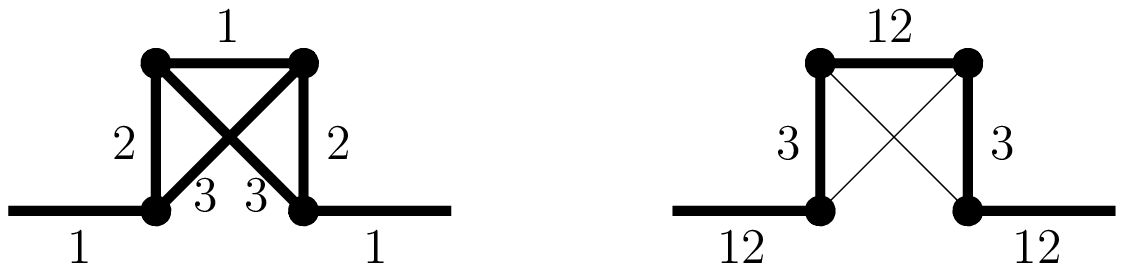}
\caption{Perfect matchings covering a copy of $K_4$ glued on some edge
of $G$. Edges are labelled $i$ if they are covered by $M_i'$ in $G'$.}\label{fig:m3}
\end{figure}

If $G$ is not
3-edge-colorable, then for any 3 perfect matchings of $G'$ there is an
edge of $G$ covered at least twice by their restrictions to $G$. This
implies that at most 4 edges are covered in each copy of $K_4$
glued on this edge (see Figure~\ref{fig:m3}, right) and the result
follows.
\end{proof}

\begin{theorem}\label{th:m4}
For any constant $ \frac{14}{15} < \tau <1$, deciding whether a cubic
bridgeless graph $G$ satisfies $m_4(G)> \tau$ (resp. $m_4(G)\ge \tau$)
is an NP-complete problem.
\end{theorem}

\begin{proof}
Instead of reducing from 3-edge-colorability, we have to make a
small modification here. The proof proceeds by reduction from the
problem of deciding whether a cubic bridgeless graph $G$ satisfies
$m_4(G)=1$, which is an NP-complete problem~\cite{EspMaz}. We
define $a=\lfloor \tfrac{m}2\rfloor+1$, where $m$ is the number of
edges of $G$, and $b=\lfloor a\cdot 
\tfrac{6-6\tau}{15\tau-14}\rfloor$. It can be checked
that $$\tau < \tfrac{6a+14b+1}{6a+15b+1}<
\tau+\tfrac{2a}{m(6a+15b+1)}.$$ As above, given a cubic
bridgeless graph $G$, we construct $G'$ by glueing $a$ copies of $K_4$
and $b$ copies of the Petersen graph on every edge of $G$. 

If $m_4(G)=1$, then $G$ contains 4 perfect matchings such that any
edge of $G$ is covered by at least 1 of the 4 perfect matchings, and
avoided by at least 2 of the 4 perfect matchings. It follows that $G'$
contains 4 perfect matchings such that their restrictions cover 6
edges in each copy of $K_4$ (see Figure~\ref{fig:m4}, left and
center) and $14$ edges in each copy of the Petersen graph (by
Lemma~\ref{lem:petersen}) glued on the edges of $G$.

\begin{figure}[ht]
\centering
\includegraphics[width=12cm]{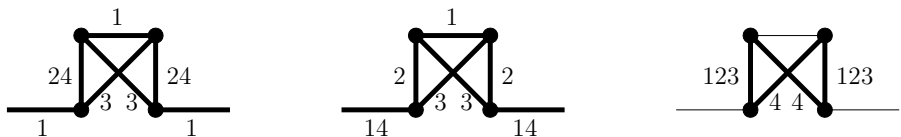}
\caption{Perfect matchings covering a copy of $K_4$ glued on some edge
of $G$. Edges are labelled $i$ if they are covered by $M_i'$ in $G'$.}\label{fig:m4}
\end{figure}

If $m_4(G)\ne 1$, then for any 4 perfect matchings of $G'$ there is an
edge of $G$ avoided by each of their restrictions to $G$. This implies
that at most 4 edges are covered in of each copy of $K_4$ glued on
this edge (see Figure~\ref{fig:m4}, right) and the result follows.
\end{proof}

Recall that the Berge-Fulkerson Conjecture is equivalent to $m_5=1$
(see~\cite{Maz}), so if this conjecture is true we cannot prove
hardness results similar to Theorems~\ref{th:m2}, \ref{th:m3},
and~\ref{th:m4}, for $m_k$ when $k\ge 5$. 

\section{Conclusion}

Let ${\cal F}_{3/5}$ be the family of cubic bridgeless graphs $G$ for
which $m_2(G)=\tfrac35$. A nice problem left open by the results in the previous section is the following.

\begin{problem}
What is the complexity of 
deciding whether a cubic bridgeless graph belongs to ${\cal
  F}_{3/5}$?
\end{problem}

By Theorem~\ref{th:m2m4}, the edge-set of a graph of ${\cal F}_{3/5}$
cannot be covered by 4 perfect matchings. Using arguments similar to
that of~\cite{KaiKraNor}, it can be proved that any graph $G\in {\cal
  F}_{3/5}$ has a set $\cal M$ of at least three perfect matchings,
such that for any $M\in \cal M$ there is a set ${\cal M}_M$ of at
least three perfect matchings of $G$ satisfying the following: for any $M'\in
{\cal M}_M$, $|M\cup M'|=\tfrac35 |E(G)|$. However, this necessary
condition is not sufficient: it is not difficult to show that it is
satisfied by the dodecahedron and by certain families of snarks.

An interesting
question is whether 
there exists any \emph{3-edge-connected} cubic bridgeless graph $G\in
{\cal F}_{3/5}$ that is different from the Petersen graph.

Similarly, we do not know if it can be decided in
polynomial time whether a cubic bridgeless graph
$G$ satisfies $m_3(G)=\tfrac45$ (resp. $m_4(G)=\tfrac{14}{15}$), but
the questions seem to be significantly harder than for $m_2$ (since
it is not known whether
$m_3=\tfrac45$ and $m_4=\tfrac{14}{15}$).

\end{document}